\documentclass[11pt, reqno]{amsart}
\usepackage[cp1251]{inputenc}
\usepackage[english]{babel}
\usepackage{amssymb}
\usepackage{amsfonts}
\usepackage{amsmath}
\usepackage{mathtools}
\usepackage{amsthm}
\usepackage{xcolor}
\usepackage{hyperref}
\usepackage{amsaddr}
\usepackage{enumitem}
\usepackage{graphicx}
\usepackage{comment}
\usepackage[noabbrev]{cleveref}
\usepackage[left=1in,right=1in,top=1in,bottom=1in]{geometry}

\usepackage{array,tabularx,tabulary,booktabs}
\usepackage{longtable}
\usepackage{multirow}
\usepackage{wrapfig}
\usepackage{subcaption}
\usepackage[table]{xcolor}
\usepackage{soul}

\newtheorem{theorem}{Theorem}
\newtheorem{corollary}{Corollary}
\newtheorem{proposition}{Proposition}
\newtheorem{problem}{Problem}

\newtheorem{conjecture}{Conjecture}

\DeclareMathOperator{\AG}{AG}
\DeclareMathOperator{\PG}{PG}
\DeclareMathOperator{\PQ}{PQ}

\DeclareMathOperator{\PGammaL}{P\Gamma L}
\DeclareMathOperator{\Aut}{Aut}

\newcommand{\ml}{\mathcal L}
\newcommand{\D}{\mathcal D}

\newcommand{\A}{\mathcal A}
\newcommand{\Z}{\mathbb Z}
\newcommand{\F}{\mathbb F}

\begin{document}

\title{Erd\H{o}s-Ko-Rado properties of Steiner 2-designs}

\author[S.~Adriaensen, S.~Goryainov, E.~V.~Konstantinova, and V.~Kr\v{c}adinac]{Sam Adriaensen$^2$, Sergey Goryainov$^1$, 
Elena V.\ Konstantinova$^{1,3,4}$, and Vedran Kr\v{c}adinac$^5$}


\address{$^1$School of Mathematical Sciences, Hebei International Joint Research Center for Mathematics and Interdisciplinary Science, Hebei Key Laboratory of Computational Mathematics and Applications, Hebei Workstation for Foreign Academicians, Hebei Normal University, Shijiazhuang 050024, P.R. China}
\address{$^2$Department of Mathematics and Data Science, Vrije Universiteit Brussel, Brussels, Belgium}
\address{$^3$Sobolev Institute of Mathematics, Ak. Koptyug av. 4, Novosibirsk, 630090, Russia}
\address{$^4$Novosibirsk State University, Pirogova str. 2, Novosibirsk, 630090, Russia}
\address{$^5$University of Zagreb Faculty of Science, Horvatovac 102a, 10000 Zagreb, Croatia}

\email{Sam.Adriaensen@vub.be}
\email{sergey.goryainov3@gmail.com}
\email{e\_konsta@math.nsc.ru}
\email{vedran.krcadinac@math.hr}



\date{September 22, 2026}

\begin{abstract}
In this paper, we prove an Erd\H{o}s-Ko-Rado characterisation of maximum intersecting families of blocks in Steiner $2$-designs arising from Desarguesian maximal arcs. This answers a recent question of Goryainov and Konstantinova, and implies that, among the known Steiner $2$-designs, only finitely many admit a maximum intersecting family that is neither canonical nor associated with a subdesign. 
We also perform a computational study of $2$-$(120,8,1)$ designs and find strong counterexamples to a problem of Godsil and Meagher. 
Finally, we give a parametric generalisation of $2$-$(66,6,1)$ designs
with tight dual arcs as non-canonical maximum intersecting families.
\end{abstract}

\maketitle


\section{Introduction}

In \cite{EKR61}, Erd\H{o}s, Ko, and Rado published their celebrated result on maximum intersecting families of $k$-subsets of an
$n$-element set, which we will abbreviate as the EKR theorem. Since then, a lot of research has been devoted to establishing analogues
of this theorem in a variety of algebraic and combinatorial settings. Let us mention some references in chronological order: \cite{Rands82},
\cite{PSV11}, \cite{GM15}, \cite{YLW22}, \cite{A22}, \cite{GK24}, \cite{GY24}, and \cite{LT24}.

A $2$-$(v,k,1)$ \emph{design} $\D$ is a collection $B$ of $k$-subsets (\emph{blocks}) of a $v$-set $V$ of \emph{points}, with the
property that every pair of points is contained in exactly one block. The parameters are sometimes written as $2$-$(v,b,r,k,1)$,
where $b=\frac{v(v-1)}{k(k-1)}$ is the total number of blocks, and $r=\frac{v-1}{k-1}$ is the number of blocks through a fixed point.
Two distinct blocks of $\D$ must have intersection of size~$0$ or~$1$. An \emph{intersecting family} of blocks is a subset
$Y\subseteq B$ such that any two blocks in $Y$ have intersection of size~1. The EKR-type question is: what is the largest possible
size of such a family~$Y$? Clearly, if we take the collection of all blocks containing a fixed point, we will have a family of size~$r$.
An EKR-type theorem for Steiner $2$-designs would state that this is the largest possible intersecting family of blocks, and determine
the conditions when the only families of maximum size are the sets of all blocks through a fixed point. The study of such questions dates
back to the work of Rands~\cite{Rands82} for general designs, and the case of Steiner $2$-designs received extra attention in the book of
Godsil and Meagher \cite[Chapter 5]{GM15}.

The \emph{block graph} of a $2$-$(v,k,1)$ design $\D=(V,B)$ is the graph with the blocks $B$ as vertices, in which two blocks are adjacent
if they have intersection of size~$1$. The block graph will be denoted by $X_{\D}$. There is a one-to-one correspondence between cliques
in~$X_{\D}$ and intersecting families of blocks in~$\D$. By Fisher's inequality, the number $b$ of blocks is at least $v$,
which is equivalent to $v\ge k^2-k+1$ for Steiner $2$-designs. Equality is reached for projective planes, where any two blocks have
non-empty intersection, and~$X_{\D}$ is the complete graph. To avoid this trivial case, we consider designs with $v>k^2-k+1$, 
or equivalently $v\ge k^2$.

A simple combinatorial argument shows that the block graph of a $2$-$(v,k,1)$ design is strongly regular (see \cite[Theorem 5.3.1]{GM15}).
The Delsarte bound~\cite{PD73} gives $r=\frac{v-1}{k-1}$ as the maximum size of cliques in $X_\D$, and thus establishes the EKR-bound for
Steiner $2$-designs. The maximum cliques consisting of all blocks through a given point are called \emph{canonical}.
In \cite[Theorem~5.1]{D15} and \cite[Corollary 5.3.5]{GM15}, it was proved that all maximum cliques in the block graph of a
$2$-$(v,k,1)$ design are canonical if $v > k^3 -2k^2+2k$ (the result was already stated in~\cite{Rands82}). However, for
\begin{equation}\label{ineq:CharacterisationFailsDesigns}
k^2\le v \le k^3 -2k^2+2k
\end{equation}
the characterisation may fail. Designs with non-canonical maximum cliques are easy to find when either
the left or the right inequality is attained.

For $v=k^2$, the design is an affine plane with $b=k(k+1)$ blocks partitioned into~$k+1$ parallel classes of size~$k$.
By choosing one block from each parallel class, one gets an intersecting family of maximum size $r=k+1$. Hence, the
block graph contains exactly $k^{k+1}$ maximum cliques, of which $k^2$ are canonical. This example is analogous
to the case $n=2k$ of the classic EKR theorem, where the $N=\binom{2k}{k}$ subsets are partitioned into disjoint pairs,
and there are exactly $2^{N/2}$ maximum intersecting families.

Let $\PG_d(n,q)$ be the design of points and $d$-dimensional subspaces as blocks in the $n$-dimensional projective
space over the finite field~$\F_q$. Then, $\D=\PG_1(3,q)$ is a $2$-$(q^3+q^2+q+1,q+1,1)$ design satisfying $v = k^3 -2k^2+2k$ with non-canonical maximum cliques in its block graph. Indeed, any
subplane of $\PG_1(3,q)$ contains $r=q^2+q+1$ blocks, any two of which have a non-empty intersection. Thus, there are $q^3+q^2+q+1$ non-canonical maximum cliques in~$X_\D$, as many as there are canonical ones.


De Boeck \cite[Corollary 5.6]{D15} improved the bound $v > k^3 -2k^2+2k$ for designs in which all maximum 
cliques are canonical. He proved that if~$\D$ is a $2$-$(v,k,1)$ design with $k \ge 4$ and
\begin{equation}\label{DeBoeckGE}
v \ge k^3-4 k^2+\frac{3}{4} \sqrt{k}(k-1)+5 k-1,
\end{equation}
then all maximum cliques in $X_\D$ are canonical, unless $v = k^3-2k^2+2k$ or~$\D$ is a
$2$-$(25,4,1)$ design. All $2$-$(25,4,1)$ designs are known~\cite{ES96}, and none of them have
non-canonical maximum cliques. For $k=3$, only $v=13$ is admissible
and satisfies strict inequality in~\eqref{ineq:CharacterisationFailsDesigns}. There are two $2$-$(13,3,1)$ designs up to isomorphism, and both have only canonical maximum cliques. Thus, $2$-$(v,k,1)$ designs satisfying~\eqref{DeBoeckGE} may have non-canonical cliques only for $v = k^3-2k^2+2k$.

Notice that there are designs with $v = k^3-2k^2+2k$ different from $\PG_1(3,q)$.
Permuting the points of a subplane of $\PG_1(3,q)$ gives a design with the same parameters, which is generally not
isomorphic~\cite[Lemma~1.1]{JT10}. Furthermore, designs with parameters $2$-$(q^3+q^2+q+1,q+1,1)$ exist
for some~$q$ that are not prime powers. For example, $2$-$(259,7,1)$ designs exist according to
\cite[Table~II.3.6 on p.~73]{CD07}. By~\cite[Exercise 5.7]{GM15}, any non-canonical maximum clique in a
$2$-$(v,k,1)$ design with $v = k^3-2k^2+2k$ comes from a subdesign with parameters $2$-$(k^2-k+1,k,1)$,
which is a projective plane of order $q=k-1$. In particular, $2$-$(259,7,1)$ designs only have
canonical maximum cliques because projective planes of order~$6$ do not exist. Godsil and Meagher
posed the following problem in their book.

\begin{problem}[{\cite[Problem 16.3.2]{GM15}}]\label{prob:GMNonCanonical}
When the block graph of a design  has maximum cliques
that are not canonical, are the non-canonical cliques isomorphic to smaller designs?
\end{problem}

Recently, it was found in~\cite{GK24} that a $2$-$(66,6,1)$ design constructed by Denniston~\cite{D80}
gives a negative answer to Problem~\ref{prob:GMNonCanonical}. For a maximum clique~$Y$ in~$X_\D$,
we define an incidence structure~$\ml(Y)$ where the points are blocks of~$Y$ and the lines are points of~$\D$ contained in at least two blocks of~$Y$. Since any two blocks of~$Y$ intersect in one point, 
any two points of~$\ml(Y)$ are incident with one line; therefore, $\ml(Y)$ is a finite linear space.
If~$Y$ is canonical, then~$\ml(Y)$ is the trivial linear space with all points on a single line.
For the non-canonical maximum cliques in the $2$-$(66,6,1)$ design~\cite{D80}, $\ml(Y)$ is
either a $2$-$(13,4,1)$ design (a projective plane of order~$3$, which is self-dual), or a $2$-$(13,3,1)$
design. The $2$-$(13,3,1)$ design is not a subdesign in the strict sense, because it is dually embedded
in the $2$-$(66,6,1)$ design (points of the small design are blocks of the large design and vice
versa). Hence, this design is a ``weak counterexample'' for Problem~\ref{prob:GMNonCanonical}.
If the problem is interpreted in a way that includes isomorphism
and duality, then it can still be considered open: so far there have been no ``strong 
counterexamples''. Barring affine planes, the following problem is open in either sense.

\begin{problem}[{\cite[Problem 2]{GK24}}]\label{prob:NonCanonicalFamiliesWithoutDesignStructure}
Does there exist an infinite family of $2$-designs whose block graphs
have non-canonical maximum cliques without a design structure?
\end{problem}

In \cite[Table~II.3.3 on p.~72]{CD07}, the known existence results for 2-$(v,k,1)$ designs with $k \le 9$ are given.
For $k \ge 10$, much less is known. However, as was discussed above, non-canonical maximum cliques in the block
graph of a 2-$(v,k,1)$ design may exist only if $v = k^3-2k^2+2k$ or
\begin{equation}\label{DeBoeckLT}
k^2 < v < k^3-4 k^2+\frac{3}{4} \sqrt{k}(k-1)+5 k-1,
\end{equation}
again barring affine planes. This means that for a fixed value of $k$, there can exist only finitely 
many $2$-$(v,k,1)$ designs with non-canonical maximum cliques. For $v = k^3-2k^2+2k$, we know that all non-canonical maximum cliques come from subdesigns~\cite[Exercise 5.7]{GM15}.

According to \cite[Theorem~II.5.11 on p.~103]{CD07}, there are only four families of parameters for which
Steiner $2$-designs are known to exist for infinitely many values of~$k$. Among them, only certain 
Denniston designs~\cite{D69} with parameters of the form
$$2\mbox{-}(2^{s+t}+2^s-2^t,2^s,1), \kern 2mm 2 \le s < t$$
satisfy inequality~\eqref{DeBoeckLT}. It can be shown by direct substitution that the 
inequality holds if and only if $t < 2s$. We thus consider the following open problem.

\begin{problem}[{\cite[Problem 3]{GK24}}]\label{prob:DennistonDesigns}
Does there exist a Denniston design whose block graph has a non-canonical maximum clique?
\end{problem}

Note that Denniston designs come from maximal arcs in the Desarguesian projective planes $\PG(2,2^t)$.
Apart from Denniston's maximal arcs~\cite{D69}, there are other constructions of maximal arcs 
in $\PG(2,2^t)$, e.g.~\cite{M02}. In this paper, we solve Problem~\ref{prob:DennistonDesigns}
by proving that designs coming from maximal arcs in Desarguesian projective planes contain
only canonical maximum cliques. Furthermore, we provide strong counterexamples
for Problem~\ref{prob:GMNonCanonical}, i.e.\ Steiner $2$-designs with non-canonical maximal 
cliques~$Y$ such that the linear space $\ml(Y)$ has non-constant point and line degrees,
and therefore $Y$ is neither a subdesign nor a dual subdesign. Problem~\ref{prob:NonCanonicalFamiliesWithoutDesignStructure}
remains open, but we suggest a family of suitable parameters that could provide infinitely 
many weak counterexamples.

The paper is organised as follows. In Section \ref{sec:ProofThm1}, we give preliminary 
definitions and results about arcs in finite projective planes. We prove 
Theorem~\ref{thm:EKRDesarguesianArcs}, leading to the EKR-characterisation of maximum 
cliques in Denniston designs. In the sequel, we perform a computational study of the smallest
non-trivial case, which occurs for order~$16$. All maximal arcs in the~$22$ known projective 
planes of order~$16$ give rise to designs without non-canonical maximum cliques, suggesting 
that the conclusion of Theorem~\ref{thm:EKRDesarguesianArcs} could hold even for 
non-Desarguesian planes. We find an example of a $2$-$(120,8,1)$ design with non-canonical 
maximum cliques~$Y$ which is not embedded in a projective plane of order~$16$.
We construct many more such designs and provide detailed statistics for the linear 
spaces~$\ml(Y)$, giving strong counterexamples for Problem~\ref{prob:GMNonCanonical}.

In Section \ref{sec:TightDualArcs}, we give a concise description of 
three $2$-$(66,6,1)$ designs and their non-canonical maximum cliques.
It turns out that the non-canonical maximum cliques stem from tight 
dual arcs, which we survey briefly and prove a new characterisation of
in Proposition~\ref{TightArcs}. We give a series of admissible parameters~\eqref{eqseries}
for Steiner $2$-designs $\D$ with tight dual arcs as non-canonical maximum 
cliques, involving two arbitrary integers~$s$ and~$t$. These designs
exist for $s=2$ and infinitely many values of~$t$, but in this case
$\D$ is an affine plane. The smallest non-trivial case are the
$2$-$(66,6,1)$ designs, occurring for $s=3$ and $t=2$. Existence
remains an open question for other values of $s\ge 3$ and~$t$, 
but this series of parameters could potentially hold an infinite 
family of examples for Problem~\ref{prob:NonCanonicalFamiliesWithoutDesignStructure}.

\section{Maximum cliques in Denniston designs}\label{sec:ProofThm1}

We start by reviewing basic definitions and results about arcs in finite projective planes. Proofs can be found in \cite[Chapter VIII.5]{BJL99} and in~\cite{JH98}. A $(n,d)$-\emph{arc} ($n, d > 1$) in a finite projective plane $\pi$ (not necessarily Desarguesian) is a set $\A$ of $n$ points of $\pi$ such that each line intersects $\A$ in at most $d$ points, and there is at least one line that does intersect $\A$ in $d$ points. The number~$d$ is called the \emph{degree} of~$\A$. The size of an arc in a projective plane of order $q$ is bounded by $n\le 1+(q+1)(d-1)$. When equality occurs, one calls~$\A$ a \emph{maximal arc}, and each line of $\pi$ meets~$\A$ in $0$ or $d$ points. The former lines are called \emph{exterior lines} and the latter are called \emph{secants} of~$\A$. Moreover, if $d<q+1$, then $d$ must divide $q$. Writing $q=d d'$, the exterior lines of~$\A$ form a maximal $(n',d')$-arc in the dual plane~$\pi'$, called the \emph{dual arc} of~$\A$ and denoted by~$\A'$. In the special case $d=2$, maximal arcs are called \emph{hyperovals}. The points of a maximal arc $\A$, and the intersections of $\A$ with its secants as blocks, give rise to a 2-$(n,d,1)$ design.


The following theorem gives a characterisation of lines in $\PG(2,q)$.
\begin{theorem}[{\cite{BW87}}]
 Let $L$ be a set of $q+1$ points in $\PG(2,q)$.
 Suppose that there are at least $q$ points $P \notin L$ such that every line through $P$ meets $L$ in a unique point.
 Then, $L$ is the set of points of a line.
\end{theorem}

The following important corollary immediately follows.

\begin{corollary}\label{Crl:Blokhuis}
  Let $\ml$ be a set of $q+1$ lines in $\PG(2,q)$.
 Suppose that there are at least $q$ lines $\ell \notin \ml$ such that every line of $\ml$ intersects $\ell$ in a different point.
 Then $\ml$ consists of the set of lines through some point.
\end{corollary}

We now prove our main result, giving a negative answer to Problem~\ref{prob:DennistonDesigns}.

\begin{theorem}\label{thm:EKRDesarguesianArcs}
 Let $\A$ be a maximal arc in $\PG(2,q)$. Then the largest intersecting families of blocks in the $2$-design obtained from $\A$ are the canonical ones.
\end{theorem}

\begin{proof}
Let $\A$ be a maximal $(n,d)$-arc in $\PG(2,q)$.
 Take an intersecting family of size $q+1$ in the associated 2-$(n,d,1)$ design, and let $\ml$ denote the lines of $\PG(2,q)$ corresponding to the blocks of this intersecting family.
 We need to prove that there exists some point $P$ such that $\ml$ consists exactly of the lines going through $P$.
 Since $\ml$ represents an intersecting family, this implies that $P \in \A$ and the result follows.

 Take a line $\ell$ exterior to $\ml$.
 Then $\ell$ must intersect every line of $\ml$ in a distinct point.
 Otherwise, there would be two lines $\ell_1, \ell_2 \in \ml$ that intersect in a point $R \in \ell$.
 But since $\ell$ is exterior to $\A$, this contradicts that $\ell_1$ and $\ell_2$ represent intersecting blocks of the design.
 Since the exterior lines to $\A$ form a maximal arc in the dual plane of $\PG(2,q)$ of degree $q/d$, there are at least $q+2$ exterior lines.
 Therefore, $\ml$ meets the conditions of \Cref{Crl:Blokhuis}, and must thus consist of the $q+1$ lines through some point $P$.
\end{proof}

In~\cite{PRS96}, all hyperovals in the 22 known projective planes of order $16$ were determined.
The dual of each hyperoval gives rise to a 2-(120,8,1) design; a total of 93 non-isomorphic
designs $\D_1,\ldots,\D_{93}$ are obtained. Using the computer program Cliquer~\cite{NO03},
we checked that the block graphs of these designs contain only canonical maximum cliques. In view
of this, we conjecture that the conclusion of Theorem~\ref{thm:EKRDesarguesianArcs} also holds
in non-Desarguesian projective planes.

\begin{conjecture}
Let $\A$ be a maximal $(n,d)$-arc in an arbitrary projective plane of order~$q$. Then the
largest intersecting families in the corresponding $2$-$(n,d,1)$ design are the canonical ones.
\end{conjecture}

A $2$-$(120,8,1)$ design $\D_0$ not embedded in a projective plane 
was constructed in~\cite{BHT97} by extending the partial geometry $\PQ^+(2n-1,2)$,
$n=2$ from~\cite{DCDT80}. The automorphism group $\Aut(\D_0)$ is isomorphic to
$(C_2)^4 \rtimes A_8$ of order $322560$, while the largest automorphism group of
the designs $\D_1,\ldots,\D_{93}$ is $\Aut(\D_1)\cong \PGammaL(2,16)$ of
order $16320$ ($\D_1$ is the dual arc of the non-degenerate conic in $\PG(2,16)$
together with its nucleus). In~\cite[p.~40]{BHT97} it was stated that several
non-isomorphic $2$-$(120,8,1)$ designs could possibly be obtained by extending the
same partial geometry $\PQ^+(3,2)$. Indeed, we constructed three more such
extensions $\D'_0$, $\D''_0$, and $\D'''_0$ with automorphism groups of orders
$720$, $576$, and $360$, respectively.

By using Cliquer~\cite{NO03}, we found that the block graph of $\D_0$ does contain non-canonical
maximum cliques, while the block graphs of $\D'_0$, $\D''_0$, $\D'''_0$ contain only canonical
maximum cliques. A few more $2$-$(120,8,1)$ designs were recently constructed in~\cite{IH26} by
prescribing groups of automorphisms. We were able to construct many more non-isomorphic examples
by the following two methods.
\begin{itemize}
\item Using GAP~\cite{GAP25} and the PAG package~\cite{PAG}, we computed subgroups of the
full automorphism groups of known designs. We then used suitable subgroups as prescribed
groups of automorphisms to construct more designs. We employed the computational method
described in~\cite{KKTVKW25}.
\item A switching method for Steiner $2$-designs called \emph{paramodification} was
introduced in~\cite{MN21}. We used our implementation of paramodification based on
Cliquer~\cite{NO03} to construct many new $2$-$(120,8,1)$ designs from old ones,
including some with trivial full automorphism groups.
\end{itemize}

A summary of our computations is given in the next theorem.

\begin{theorem}\label{thm120-8-1}
There are at least $391733$ non-isomorphic $2$-$(120,8,1)$ designs. Their distribution
by orders of full automorphism groups is given in Table~\ref{tab2}.
\end{theorem}

\begin{table}[!ht]
\begin{center}
\begin{tabular}{|cc|cc|cc|cc|cc|cc|}
\hline
\rule{0mm}{9pt}$|\Aut|$ & $\#$ & $|\Aut|$ & $\#$ & $|\Aut|$ & $\#$ & $|\Aut|$ & $\#$ & $|\Aut|$ & $\#$ & $|\Aut|$ & $\#$ \\
\hline
\rule{0mm}{9pt}322560 & 1 & 1024 & 3 & 384 & 287 & 168 & 3 & 48 & 752 & 8 & 76497\\
16320 & 1 & 896 & 1 & 360 & 1 & 144 & 3 & 42 & 22 & 6 & 489 \\
9216 & 1 & 768 & 32 & 336 & 16 & 128 & 8419 & 32 & 98392 & 4 & 44027 \\
3072 & 3 & 720 & 1 & 320 & 1 & 112 & 2 & 24 & 697 & 3 & 1 \\
2688 & 6 & 576 & 2 & 288 & 5 & 96 & 574 & 16 & 103616 & 2 & 17214 \\
1536 & 8 & 512 & 88 & 256 & 768 & 80 & 1 & 14 & 2 & 1 & 3 \\
1152 & 4 & 448 & 1 & 192 & 315 & 64 & 38495 & 12 & 979 & & \\
\hline
\end{tabular}
\end{center}
\caption{Distribution of the known $2$-$(120,8,1)$ designs by $|\Aut(\D)|$.}\label{tab2}
\end{table}

A GAP-readable file containing a list of all the constructed designs is available on the
web page~\cite{VKSteiner}. The list is sorted by decreasing orders of automorphism
groups $|\Aut(\D)|$. The designs $\D_1,\ldots,\D_{93}$ appear in the list at positions
2, 457, 1549, 9971, 9972, 10547-10555, 49817-49819, 148906-148964,
252522, 252523, 253503-253507, 330000, 330001, 330489-330492, 374516-374519,
and the designs $\D_0$, $\D'_0$, $\D''_0$, $\D'''_0$ appear at positions 1, 61, 62, 440.
We analysed the block graphs of the constructed $2$-$(120,8,1)$ designs
and found that the majority do have non-canonical maximum cliques.

\begin{proposition}
For the $2$-$(120,8,1)$ designs of Theorem~\ref{thm120-8-1}, $362024$ out of $391733$ block graphs contain non-canonical maximum cliques~$Y$. Three different non-trivial linear spaces
$\ml_1$, $\ml_2$, $\ml_3$ occur as~$\ml(Y)$; they are described in Table~\ref{tab3}.
\end{proposition}

\begin{table}[t]
\begin{center}
\begin{tabular}{|cccccc|}
\hline
\rule{0mm}{10pt}Lin.\ sp. & No.\ of points & Point degrees & No.\ of lines & Line degrees & $|\Aut(\ml_i)|$ \\[0.5mm]
\hline
\rule{0mm}{10pt}$\ml_1$ & 17 & $\{4^1, 5^{16}\}$ & 20 & $\{4^{16}, 5^4\}$ & 1152 \\[0.5mm]
$\ml_2$ & 17 & $\{7^{16}, 8^1\}$ & 40 & $\{2^{16}, 3^8, 4^{16} \}$ & 384 \\[0.5mm]
$\ml_3$ & 17 & $\{5^4, 7^{10}, 8^3\}$ & 38 & $\{2^{16}, 3^9, 4^{12}, 7^1\}$ & 24 \\[0.5mm]
\hline
\end{tabular}
\end{center}
\caption{Linear spaces occurring as non-canonical maximum cliques in $2$-$(120,8,1)$ designs.}\label{tab3}
\end{table}

\begin{table}[b!]
\begin{center}
\begin{tabular}{|cc|cc|cc|cc|cc|}
\hline
\rule{0mm}{9pt}$(n_1,n_2,n_3)$ & $\#$ & $(n_1,n_2,n_3)$ & $\#$ & $(n_1,n_2,n_3)$ & $\#$ & $(n_1,n_2,n_3)$ & $\#$ & $(n_1,n_2,n_3)$ & $\#$ \\
\hline
\rule{0mm}{9pt}$(840, 0, 0)$ & 1 & $(268, 0, 0)$ & 1 & $(160, 0, 0)$ & 705 & $(90, 0, 0)$ & 1 & $(42, 0, 0)$ & 562 \\
$(616, 0, 0)$ & 1 & $(264, 0, 0)$ & 36 & $(156, 0, 0)$ & 9 & $(88, 0, 0)$ & 3045 & $(40, 0, 0)$ & 117 \\
$(488, 0, 0)$ & 1 & $(260, 0, 0)$ & 2 & $(152, 0, 0)$ & 2178 & $(86, 0, 0)$ & 2 & $(36, 0, 0)$ & 239 \\
$(472, 0, 0)$ & 1 & $(256, 0, 0)$ & 22 & $(148, 0, 0)$ & 51 & $(84, 0, 0)$ & 16 & $(32, 0, 0)$ & 1580 \\
$(456, 0, 0)$ & 1 & $(252, 0, 0)$ & 13 & $(146, 0, 0)$ & 1 & $(82, 0, 0)$ & 3 & $(30, 0, 0)$ & 4 \\
$(448, 0, 0)$ & 2 & $(248, 0, 0)$ & 119 & $(144, 0, 0)$ & 2564 & $(80, 0, 0)$ & 4074 & $(28, 0, 0)$ & 193 \\
$(440, 0, 0)$ & 3 & $(244, 0, 0)$ & 4 & $(142, 0, 0)$ & 1 & $(78, 0, 0)$ & 2 & $(26, 0, 0)$ & 3 \\
$(424, 0, 0)$ & 2 & $(240, 0, 0)$ & 209 & $(140, 0, 0)$ & 5 & $(76, 0, 0)$ & 68 & $(24, 0, 0)$ & 5865 \\
$(408, 0, 0)$ & 2 & $(236, 0, 0)$ & 4 & $(136, 0, 0)$ & 1473 & $(74, 0, 0)$ & 12 & $(22, 0, 0)$ & 5 \\
$(400, 0, 0)$ & 2 & $(232, 0, 0)$ & 89 & $(132, 0, 0)$ & 3 & $(72, 0, 0)$ & 17675 & $(20, 0, 0)$ & 2538 \\
$(376, 0, 0)$ & 8 & $(228, 0, 0)$ & 1 & $(130, 0, 0)$ & 1 & $(68, 0, 0)$ & 142 & $(18, 0, 0)$ & 6 \\
$(360, 0, 0)$ & 4 & $(224, 0, 0)$ & 81 & $(128, 0, 0)$ & 185 & $(66, 0, 0)$ & 11 & $(16, 0, 0)$ & 2124 \\
$(352, 0, 0)$ & 7 & $(216, 0, 0)$ & 195 & $(124, 0, 0)$ & 4 & $(64, 0, 0)$ & 17440 & $(14, 0, 0)$ & 31 \\
$(344, 0, 0)$ & 14 & $(212, 0, 0)$ & 4 & $(122, 0, 0)$ & 1 & $(62, 0, 0)$ & 5 & $(12, 0, 0)$ & 7272 \\
$(340, 0, 0)$ & 1 & $(208, 0, 0)$ & 99 & $(120, 0, 0)$ & 1531 & $(60, 0, 0)$ & 174 & $(10, 0, 0)$ & 28 \\
$(336, 0, 0)$ & 42 & $(200, 0, 0)$ & 427 & $(118, 0, 0)$ & 1 & $(58, 0, 0)$ & 2 & $(8, 48, 0)$ & 2037 \\
$(328, 0, 0)$ & 2 & $(196, 0, 0)$ & 2 & $(116, 0, 0)$ & 2 & $(56, 0, 64)$ & 42 & $(8, 24, 0)$ & 54 \\
$(320, 0, 0)$ & 15 & $(192, 0, 0)$ & 717 & $(112, 0, 0)$ & 344 & $(56, 0, 32)$ & 30 & $(8, 12, 0)$ & 3 \\
$(312, 0, 0)$ & 13 & $(188, 0, 0)$ & 12 & $(110, 0, 0)$ & 2 & $(56, 0, 16)$ & 60 & $(8, 0, 0)$ & 15644 \\
$(304, 0, 0)$ & 1 & $(184, 0, 0)$ & 452 & $(108, 0, 0)$ & 8 & $(56, 0, 0)$ & 257477 & $(4, 0, 0)$ & 114 \\
$(296, 0, 0)$ & 14 & $(180, 0, 0)$ & 1 & $(104, 0, 0)$ & 847 & $(54, 0, 0)$ & 14 & $(2, 0, 0)$ & 134 \\
$(288, 0, 0)$ & 12 & $(178, 0, 0)$ & 1 & $(100, 0, 0)$ & 7 & $(52, 0, 0)$ & 63 & $(0, 48, 0)$ & 49 \\
$(284, 0, 0)$ & 3 & $(176, 0, 0)$ & 730 & $(98, 0, 0)$ & 3 & $(50, 0, 0)$ & 131 & $(0, 24, 0)$ & 10 \\
$(280, 0, 0)$ & 24 & $(168, 0, 0)$ & 1055 & $(96, 0, 0)$ & 618 & $(48, 0, 0)$ & 7771 & $(0, 0, 0)$ & 29709 \\
$(272, 0, 0)$ & 44 & $(164, 0, 0)$ & 2 & $(92, 0, 0)$ & 30 & $(44, 0, 0)$ & 80 & & \\
\hline
\end{tabular}
\end{center}
\caption{Distribution of the known $2$-$(120,8,1)$ designs by their vectors $(n_1,n_2,n_3)$.}\label{tab4}
\end{table}

This gives a strong negative answer to Problem~2, in the sense that the
non-canonical maximum cliques are neither designs nor duals of designs.
Every $2$-$(120,8,1)$ design has $120$ canonical maximum cliques~$Y$, leading
to the trivial linear space as~$\ml(Y)$. For each of our designs~$\D$, we 
calculated a vector $(n_1,n_2,n_3)$ containing the numbers $n_i$ of non-canonical 
maximum cliques~$Y$ in~$X_\D$ such that $\ml(Y)=\ml_i$, $i\in\{1,2,3\}$. For 
example, the block graph $X_{\D_0}$ contains exactly $840$ maximum cliques with 
$\ml(Y)=\ml_1$ and no maximum cliques with $\ml(Y)=\ml_2$ or $\ml(Y)=\ml_3$, 
so the vector of $\D_0$ is $(840,0,0)$. The distribution of all our designs by 
these vectors is given in Table~\ref{tab4}.

In view of these computational results, it appears that candidates for other
designs with non-canonical maximum cliques could be designs with parameters
as Denniston designs, but not embedded in projective planes. The next few
cases are $2$-$(232,8,1)$, $2$-$(496,16,1)$, and $2$-$(976,16,1)$. We have
not succeeded in constructing examples of these designs possessing non-canonical maximum
cliques.

\section{Tight dual arcs as non-canonical maximum cliques}\label{sec:TightDualArcs}

A $2$-$(66,6,1)$ design was constructed by Denniston~\cite{D80} by extending a
certain $2$-$(40,4,1)$ design of Lorimer~\cite{PL74}. In~\cite{VK04}, the $2$-$(66,6,1)$ designs
with automorphisms of order~$13$ were classified and two more examples were found.
It was noted in~\cite{GK24} that these three designs contain non-canonical
maximum cliques in their block graphs. A fourth $2$-$(66,6,1)$ design with an automorphism
group of order~$20$ was recently constructed in~\cite{BHR25}, but this design contains
only canonical maximum cliques.

We give a short description of the three known $2$-$(66,6,1)$ designs with non-canonical
maximum cliques, in the style of~\cite[Table~II.3.32 on p.~76]{CD07}. The three designs
$\D_{66}^{(i)}$, $i\in\{1,2,3\}$ are invariant under a non-Abelian group~$G$ of
order~$39$ acting on the point set $V=A\cup B\cup C\cup \{\infty\}$. Here,
$A=\Z_{13}\times \{a\}$, $B=\Z_{13}\times \{b\}$, $C=\Z_{13}\times \Z_{3}$, and
points from these sets will be written using subscripts, e.g.\ as $0_a$, $7_b$, $12_0$.
The group $G=\langle \alpha, \beta\rangle$ is generated by a permutation $\alpha$
of order~$3$ and a permutation $\beta$ of order~$13$.
The generator $\alpha$ comprises cycles $(x_0, 3x_1, 9x_2)$ and $(x_t, 3x_t, 9x_t)$,
while $\beta$ comprises cycles $(0_s, 1_s,\ldots, 12_s)$, for $x\in \Z_{13}$, $t\in \{a,b\}$,
and $s\in \Z_3\cup\{a,b\}$. The sets $A$, $B$, $C$, and $\{\infty\}$ are point-orbits of the
designs. Each design $\D_{66}^{(i)}$ has two block-orbits $(S_j^{(i)})^G$ of
size~$13$, $j\in\{1,2\}$, and three block-orbits $(L_j^{(i)})^G$ of size~$39$,
$j\in\{1,2,3\}$. The block-orbits are obtained by developing the base blocks
given in Table~\ref{tab1}.

\begin{table}[!ht]
\begin{center}
\begin{tabular}{|c|ccc|}
\cline{2-4}
\multicolumn{1}{c|}{} & \rule{0mm}{11pt}$i=1$ & $i=2$ & $i=3$\\[0.5mm]
\hline
\rule{0mm}{13pt}$S_1^{(i)}$ & $\{\infty,0_a,0_b,0_0,0_1,0_2\}$ & $\{\infty,0_a,0_b,0_0,0_1,0_2\}$ & $\{\infty,0_a,0_b,0_0,0_1,0_2\}$ \\[1mm]
$S_2^{(i)}$ & $\{0_a, 1_a, 10_a, 5_b, 7_b, 12_b\}$ & $\{0_a, 1_a, 10_a, 5_b, 7_b, 12_b\}$ & $\{0_a, 2_a, 8_a, 6_b, 7_b, 10_b\}$\\[1mm]
$L_1^{(i)}$ & $\{0_a, 1_b, 4_0, 6_0, 7_1, 11_2\}$ & $\{0_a, 1_b, 9_0, 12_0, 5_2, 10_2\}$ & $\{0_a, 1_b, 4_0, 12_0, 3_1, 6_2\}$\\[1mm]
$L_2^{(i)}$ & $\{0_a, 2_a, 3_0, 10_0, 5_2, 6_2\}$ & $\{0_a, 2_a, 8_0, 4_1, 9_1, 11_2\}$ & $\{0_a, 1_a, 3_0, 7_0, 10_0, 8_2\}$\\[1mm]
$L_3^{(i)}$ & $\{0_b, 1_b, 11_1, 3_2, 9_2, 12_2\}$ & $\{0_b, 1_b, 4_0, 3_2, 6_2, 12_2\}$ & $\{0_b, 2_b, 9_0, 10_0, 5_1, 12_1\}$\\[1mm]
\hline
\end{tabular}
\end{center}
\caption{Base blocks for the $2$-$(66,6,1)$ designs $\D_{66}^{(i)}$, $i=1,2,3$.}\label{tab1}
\end{table}

In~\cite{GK24}, the block graphs of the designs $\D_{66}^{(i)}$, $i\in\{1,2,3\}$ were examined
and the non-canonical maximum cliques were listed. Each of the graphs $X_{\D_{66}^{(i)}}$ possesses
$13$ maximum cliques $Y$ such that the corresponding linear space $\ml(Y)$ is a projective plane of
order~$3$, i.e.\ a $2$-$(13,4,1)$ subdesign, and one maximum clique~$Y$ such that $\ml(Y)$ is
a $2$-$(13,3,1)$ design. The latter $Y$ is the block-orbit $(S_2^{(i)})^G$ and in this case the lines of
$\ml(Y)$ are the point-orbits $A\cup B$. The complementary point-orbits $C\cup \{\infty\}$
form a maximal $4$-arc in $\D_{66}^{(i)}$, i.e.\ a $2$-$(40,4,1)$ subdesign. For Denniston's
design $\D_{66}^{(1)}$ and for $\D_{66}^{(3)}$, the $2$-$(40,4,1)$ subdesign is isomorphic
to Lorimer's design~\cite{PL74} with automorphism group of order $151632$, while for
$\D_{66}^{(2)}$ it is another $2$-$(40,4,1)$ design with~$G$ as full automorphism group.
For all three designs $\D_{66}^{(i)}$, the maximum cliques such that $\ml(Y)$ is a projective
plane of order~$3$ are exactly the non-canonical maximum cliques of the $2$-$(40,4,1)$ subdesign.

Maximal arcs were generalised from projective planes to arbitrary 2-designs by
E.~J.~Morgan~\cite{EJM77} and Tran Van Trung~\cite{TVT91}. We give here
an overview of some of their results in the special case $\lambda=1$, i.e.\ for Steiner
2-designs.

\begin{proposition}[{\cite[Lemma~1.2]{TVT91}}]\label{ArcBoundBIBD}
If $\A$ is a $(n,d)$-arc in a $2$-$(v,b,r,k,1)$ design, then $n\le 1+r(d-1)$.
Equality holds if and only if each block meets $\A$ in $0$ or $d$ points.
\end{proposition}

\begin{proposition}[{\cite[Corollary~3.2]{EJM77}}]\label{MaxArcsBIBD}
If there is a maximal $(n,d)$-arc $\A$ in a $2$-$(v,b,r,k,1)$ design $\D$,
then $d$ divides $r-1$. The points of $\A$ and its secants as blocks
constitute a $2$-$(n,d,1)$ design.
\end{proposition}

\begin{proposition}[{\cite[Corollary~3.3]{EJM77}}]\label{DualArcsBIBD}
If there is a maximal $(n,d)$-arc in a $2$-$(v,b,r,k,1)$ design, then
there is a maximal $(n',d')$-arc in its dual, where $d'=(r-1)/d$ and $n'=b-nr/d$.
\end{proposition}

The dual $\D'$ of a $2$-$(v,b,r,k,1)$ design $\D$ is the incidence structure with
points and blocks reversed. Notice that $\D'$ is not a design, unless $\D$ is a projective plane.
A maximal $(n',d')$-arc in $\D'$ is a set~$\A'$ of~$n'$ points of~$\D'$ (blocks of~$\D$)
such that every block of~$\D'$ (point of~$\D$) is incident with~$0$ or~$d'$ elements of~$\A'$. In~$\D'$, there are~$k$ blocks through every point and one could expect the size
of a maximal $(n',d')$-arc to be $n'=1+k(d'-1)$, but this is not always true. The next
proposition characterises when it does hold.

\begin{proposition}[{\cite[Theorem~3.5]{EJM77}}]\label{TightArcsBIBD}
Let $\A$ be a maximal $(n,d)$-arc in a $2$-$(v,b,r,k,1)$ design $\D$,
and $\A'$ the corresponding maximal $(n',d')$-arc in the dual $\D'$.
Then $n'=1+k(d'-1)$ if and only if either (a) or (b) holds:
\begin{enumerate}
\item[(a)] $v=b$ and the design $\D$ is a projective plane,
\item[(b)] $d=k-k/d'$, so $d'$ divides $k$ and the design $\D$ has parameters
$v=k(d'k-k-d'+2)$, $b=(d'k-k-d'+2)(d'k-k+1)$, and $r=d'k-k+1$.
\end{enumerate}
\end{proposition}

When the equality $n'=1+k(d'-1)$ holds, we shall refer to the dual
arc~$\A'$ as \emph{tight}. Here is another characterisation of tight
dual arcs.

\begin{proposition}\label{TightArcs}
Let $\A$ be a maximal $(n,d)$-arc in a $2$-$(v,b,r,k,1)$ design $\D$,
and $\A'$ the corresponding dual $(n',d')$-arc.
Then $n'=1+k(d'-1)$ holds if and only if the points
of $\A'$ and its secants as blocks constitute a $2$-$(n',d',1)$ design,
which is a non-canonical maximal clique in~$\D$.
\end{proposition}

\begin{proof}
By counting incidences, $\A'$ has $s=n'k/d'$ secants. The points of $\A'$
are the $n'$ exterior lines of $\A$. They make up a design if and only if
any two exterior lines intersect. This is clearly equivalent to
$$\binom{n'}{2} = s \cdot \binom{d'}{2},$$
which is in turn equivalent to $n'=1+k(d'-1)$.
\end{proof}

Notice that the maximum clique $Y$ in the block graph of $\D_{66}^{(i)}$
such that $\ml(Y)$ is a $2$-$(13,3,1)$ design is in fact
a tight dual $(13,3)$-arc $\A'$. The corresponding $(40,4)$-arc
$\A$ is the $2$-$(40,4,1)$ subdesign of $\D_{66}^{(i)}$.
We can generalise this example by assuming that a $2$-$(v,b,r,k,1)$
design $\D$ contains a tight dual $(n',d')$-arc $\A'$ of size $n'=r$.
Then, $Y=\A'$ is a non-canonical maximum clique in $X_{\D}$ such
that $\ml(Y)$ is a $2$-$(r,d',1)$ design. By
Proposition~\ref{TightArcsBIBD}~(b), $d'$ divides~$k$. Denoting
$s=d'$ and $t=k/d'$, we can write all other parameters in terms
of $s$ and $t$:
\begin{equation}\label{eqseries}
\begin{array}{r@{\,\,}c@{\,\,}l}
v & = & st(st(s-1)-s+2),\\[2mm]
b & = & (st(s-1)+1)(st(s-1)-s+2),\\[2mm]
r & = & n' = st(s-1)+1,\\[2mm]
k & = & st,\\[2mm]
n & = & t(s-1)^2(st-1),\\[2mm]
d & = & t(s-1).\\[2mm]
\end{array}
\end{equation}

The designs $\D_{66}^{(i)}$ are the special case $s=3$, $t=2$. Such designs
also exist for $s=2$; then, $\D$ is an affine plane of
order $2t$ with a maximal $(t (2 t - 1), t)$-arc. For $t=2^e$, we can
construct such an affine plane by taking a projective plane $\pi$ of
order $q=2^{e+1}$ with a hyperoval $\mathcal{H}'$ in the dual plane~$\pi'$
(e.g.\ $\pi=\PG(2,2^{e+1})$ and $\mathcal{H}'$ a non-degenerate dual
conic together with its nucleus). Take any line
$\ell_\infty \in \mathcal{H}'$ and let $\D=\pi\setminus \ell_\infty$. The
remaining lines $\A'=\mathcal{H}'\setminus \{\ell_\infty\}$ are a tight
dual arc in~$\D$, i.e.\ a non-canonical maximum clique in $X_\D$. We already
know that the block graphs of affine planes have many non-canonical
maximum cliques, but tight dual arcs are a very restricted subclass.
For example, the block graph of $\D=\AG(2,4)$ has $4^5=1024$ maximum
cliques, $16$ of which are canonical and $48$ of which are
tight dual arcs. Similarly, by a computation in GAP~\cite{GAP25}, the
block graph of $\D=\AG(2,8)$ has $8^9=134\,217\,728$ maximum cliques,
$64$ of which are canonical and $4480$ of which are tight
dual $(9,2)$-arcs. The corresponding maximal $(28,4)$-arcs are unitals
of order~$3$ isomorphic to the smallest Ree unital~\cite{HL66}. Interestingly,
it is known that this unital cannot be embedded in a projective plane
of order~$9$ \cite{AEB81, KG86}, but it can be embedded in
$\AG(2,8)$ and $\PG(2,8)$.

For $s\ge 3$, equations~\eqref{eqseries} give an infinite series of
parameters for Steiner $2$-designs with tight dual arcs as non-canonical
maximum cliques. These designs are only known to exist in the smallest
case $s=3$, $t=2$. The next few cases are
\begin{enumerate}[label=\arabic*.]
\item $s=3$, $t=3$: $\D$ is a $2$-$(153,9,1)$ design with a maximal $(96,6)$-arc
and tight dual $(19, 3)$-arc,
\item $s=4$, $t=2$: $\D$ is a $2$-$(176,8,1)$ design with a maximal $(126,6)$-arc
and tight dual $(25, 4)$-arc,
\item $s=5$, $t=2$: $\D$ is a $2$-$(370,10,1)$ design with a maximal $(288,8)$-arc
and tight dual $(41, 5)$-arc,
\item $s=4$, $t=3$: $\D$ is a $2$-$(408,12,1)$ design with a maximal $(297,9)$-arc
and tight dual $(37,4)$-arc.
\end{enumerate}
We have not succeeded to construct or rule out any of these designs. Indeed, in the
first three cases the existence of a design with parameters as~$\D$ is an open question,
with or without maximal arcs. In the fourth case, both $2$-$(408,12,1)$ and $2$-$(297,9,1)$
designs are unknown. In conclusion, Problem~\ref{prob:NonCanonicalFamiliesWithoutDesignStructure}
remains an interesting and presumably rather difficult open problem at the intersection
of EKR-Theory and Design Theory.

\section*{Acknowledgments}

Sam Adriaensen is supported by grant 12A3Y25N of the Research Foundation Flanders (FWO).
Sergey Goryainov is supported by the Natural Science Foundation of Hebei Province (A2023205045) and the 111 Center (Grant No.D26018).
Elena V.~Konstantinova is supported by the state contract of the Sobolev Institute of Mathematics, project No.~FWNF-2026-0011 ``Algebraic and combinatorial invariants of discrete structures''. Vedran Kr\v{c}adinac is supported by the European Union -- NextGenerationEU through
the National Recovery and Resilience Plan 2021-2026 Institutional grants of University
of Zagreb Faculty of Science (IK IA 1.1.3 Impact4Math).

\section*{Declaration of A.I.\ use}

Artificial intelligence has only been used to do some minor editing of the text.
All the content of this paper is due to the authors.

\end{document}